\documentclass[12pt]{article}
\usepackage{amsmath,amsfonts,amssymb,amsthm,amscd}
\title{A note on the effective listing of complete types}
\date{\today}
\author
{Anand Pillay\thanks{Supported by NSF grants  DMS 1665035 and DMS-1760413}\\University of Notre Dame  }

\newtheorem{Theorem}{Theorem}[section]

\newtheorem{Definition}[Theorem]{Definition}
\newtheorem{Remark}[Theorem]{Remark}
\newtheorem{Lemma}[Theorem]{Lemma}

\newtheorem{Fact}[Theorem]{Fact}

\newcommand{\Q}{\mathbb Q}

\begin{document}
\maketitle

\begin{abstract} 
We use the ``geometric axioms" point of view to give an effective listing of the complete types of the theory $DCF_{0}$ of differentially closed fields of characteristic $0$. This gives another account of observations made in  \cite{Calvert} and \cite{Marker}.  

\end{abstract}

\section{Introduction}
We describe complete types in $DCF_{0}$ in a manner  which immediately yields an effective listing.   What we do here  should be considered folklore.  It is closely related to an unpublished account by Hrushovski of the existence of a model companion for $DCF_{0}$ equipped with an automorphism  (of which I could find no hard record).  A version of Hrushovski's work appears in \cite{Bustamante}, and was  later generalized to $DCF_{m,0}$, the model companion of the theory of fields of characteristic $0$ equipped with $m$ commuting derivations in \cite{Leon-Sanchez}.

The existence of an effective list of complete types of a theory $T$  is related in \cite{Calvert} to proofs of  {\em strong jump inversion}, a notion from recursive model theory.  Moreover in the same paper the authors give an elementary account of such an effective listing of complete $n$-types  types in the case of $DCF_{0}$, using the Blum axioms and induction on $n$.   In a correspondence with Knight, Marker \cite{Marker} gives another account of the effective listing of types of $DCF_{0}$ making use of the $ACC$ on radical differential ideals, and a finite injury argument.  Effective quantifier elimination for $DCF_{0}$ is of course in the background and allows one to focus on quantifier-free types in the above accounts.  There are also approaches involving computing the prime decomposition of a radical differential idea. 

What we describe here is, at least superficially, a bit more direct. (It is basically writing down details of what I said to Julia Knight in the summer of 2018 after she asked me about the matter.)
\newline
(i) We define the notion of a good pair $(V,W)$ of irreducible algebraic varieties over $\Q$ (and the good pairs form a recursive set). 
\newline
(ii) We define the notion of the $\partial$-generic type of such a good pair $(V,W)$, a complete type over $\emptyset$ of $DCF_{0}$, which is computable (uniformly) from $(V,W)$. 
\newline
(iii) Every complete type $p({\bar x})$ over $\emptyset$ (of $DCF_{0}$)  is the $\partial$-generic type of some good pair $(V,W)$ (after possibly replacing the tuple ${\bar x}$ of variables by $({\bar x},{\bar x}',..,{\bar x}^{(r)})$ for some $r$),

\vspace{5mm}
\noindent
This note is  expository, and is written partly for logicians with an interest in recursive model theory, so we give a few more details than would normally appear in a model theory paper.
We first give a little background on the theory $DCF_{0}$.  See \cite{Marker-book}, \cite{Pierce-Pillay}, \cite{Calvert} for some more details.  $L^{-}$ will denote the language of unitary rings, containing $0,1,+,-,\times$. The language $L$ of differential rings is  $L^{-}\cup\{\partial\}$, where the unary function symbol stands for the derivation.  The theory of integral domains equipped with a derivation is universally axiomatized in $L$, and it turns that that this universal theory has a model companion  which is denoted $DCF_{0}$.   $DCF_{0}$ is complete with quantifier elimination.
There are various recursive sets of axioms for $DCF_{0}$, which  yield that $DCF_{0}$ is decidable and that  quantifier elimination is effective.  One of the nice axiom systems, due to Lenore Blum, is only about differential polynomials $P(y)$ in one differential indeterminate $y$, and says about a differential field $K$ that whenever $P(y), Q(y)$ are differential polynomials over $K$ and the order of $Q$ is strictly less than the order of $P$ then the system $P(y) = 0$ and $Q(y)\neq 0$ has a solution in $K$ (so vacuously if $P(y)$ has order $0$, namely is an ordinary polynomial over $K$, then $P$ has a solution in $K$).

The so-called geometric axioms for $DCF_{0}$ were introduced in \cite{Pierce-Pillay} as something analogous to the Hrushovski geometric axioms for the model companion $ACFA$ of difference fields. Compared to the Blum axioms they are a bit more sophisticated, but the advantage is that they express succinctly fairly powerful facts about differentially closed fields (and moreover the Blum axioms are a special case). 
Also key aspects of differential algebra and the model theory of differential fields concern reducing or expressing differential algebraic properties to, or, in terms of,  algebraic properties, and the geometric axioms have a similar feature. 

We work in characteristic $0$. We now recall the naive account of elementary {\em algebraic-geometric} notions which suffices for our purposes.  Let $k<K$ be fields with $K$ algebraically closed (where possibly $k = K$).  A variety $V\subseteq K^{n}$ defined over $k$ is simply the zero set of a finite collection of polynomials $P(x_{1},..,x_{n})\in k[x_{1},..,x_{n}]$  with coefficients in $k$. The {\em Zariski topology} on $V$ is the topolology whose closed sets are subvarieties of $V$ (possibly defined over $K$). $V$ (assumed to be defined over $k$) is called $k$-irreducible if it $V$ can not be expressed as the union of two proper closed subsets, both defined over $k$. By quantifier elimination, there is a one-one correspondence between complete $n$-types over $k$ in the sense of the theory of algebraically closed fields, and varieties $V\subseteq K^{n}$ which are defined over $k$ and $k$-irreducible.  For example to such a $V$ attach the type $p_{V}$ which says that ${\bar x}\in V$ and ${\bar x}\notin W$ for each proper closed subvariety $W$ ov $V$ defined over $k$.  $p_{V}$ is called the ``generic type" of $V$ over $k$. The collection of $k$-irreducible varieties over $k$ is recursive (assuming $k$ to be a recursive field). 

Let $V\subseteq K^{n}$ be defined over $k$ and $k$-irreducible.  The (Zariski) tangent bundle $TV\subseteq K^{2n}$ is by definition defined by equations 
$\sum_{i=1,..,n}(\partial P/\partial x_{i})({\bar x}(u_{i}) = 0$ as $P$ ranges over polynomials which generate the ideal $I(V/k)$ of polynomials over $k$ vanishing on $V$.  
If $k$ is equipped with a derivation $\partial$ then we can shift $TV$ by the derivation (in a sense) and define $T_{\partial}(V)$  (also called $\tau(V)$, the first prolongation of $V$) to be the subset of $K^{2n}$ defined by the equations  $\sum_{i=1,..,n}(\partial P/\partial x_{i})({\bar x}(u_{i})+ P^{\partial}({\bar x}) = 0$  where $P$ is as before and $P^{\partial}$ is obtained from $P$ by applying the distinguished derivation $\partial$ on $k$ to the coefficients.  So note that when $\partial = 0$ on $k$, then $T_{\partial}(V)$ coincides with $TV$.  All the data defined so far belong to algebraic (rather than differential algebraic) geometry.

\vspace{2mm}
\noindent
The axioms for $DCF_{0}$ proposed in \cite{Pierce-Pillay} say (about a differential field $(k,\partial)$):
\newline
(i) $k$ is algebraically closed  (so in our above algebraic-geometric discussion we may take $k = K$), and 
\newline
(ii) If $V\subseteq k^{n}$ is an irreducible variety, and $W$ is an irreducible subvarieity of $T_{\partial}(V)$ such that the projection of $W$ onto the $\bar x$ coordinates  contains a nonempty Zariski open subset of $V$ (we say that $W$ projects generically onto $V$) and $U$ is a nonemptyset Zariski open subset of $W$ then there is some point $({\bar a},{\bar b})\in U$ such that ${\bar b} = \partial({\bar a})$. 

\vspace{2mm}
\noindent
Here and subsequently  if $\bar a = (a_{1},..,a_{n})$ is a finite tuple of elements from a differential field $(k,\partial)$ then $\partial({\bar a})$ denotes the tuple $(\partial(a_{1}),..,\partial(a_{n}))$. We may also write this as ${\bar a}'$. Likewise ${\bar a}'' = \partial^{2}({\bar a})$ etc. 

A corollary of the axioms is:
\begin{Fact} Let $(K,\partial)$ be a differentially closed field, and $k$ an arbitrary differential subfield (not necessarily algebraically closed). Let $V\subset K^{n}$ be a variety defined over $k$ and irreducible over $k$. Let $W$ be a $k$-irreducible subvariety of $T_{\partial}(V)$ defined over $k$ which projects generically on to $V$. 
\newline
Then
\newline
(i) Let $U$ be a nonempy Zariski open subset of $W$ defined over $k$. Then there is $(a,b)\in K^{2n}$ such that $(a,b)\in U$ and $\partial(a) = b$. 
\newline
(ii) Suppose moreover that $(K,\partial)$ is $|k|^{+}$-saturated. Then there is some $(a,b)\in K^{2n}$ such that $\partial(a) = b$ and $(a,b)$ is a generic point over $k$ of $W$ (in the sense of $ACF$). 
\end{Fact}

\vspace{5mm}
\noindent

 \section{Results.}

We will fix a saturated differentially closed field ${\cal U}$.   and a small differential subfield $k$, which may at one extreme  be $\Q$, and at another extreme be an elementary submodel. 

Let us first remark that by quantifier elimination in $DCF_{0}$ the complete $L$-type of a tuple ${\bar a}$ over $k$ is determined by  the information of which quantifier-free $L^{-}$ formulas with parameters from $k$ are satisfied by $({\bar a},{\bar a}')$, $({\bar a}, {\bar a}', {\bar a}'')$, etc. 

Following the strategy mentioned in the introduction we first give the notion of a good pair over $k$.  But for the purposes of listing complete types of $DCF_{0}$ only the case $k=\Q$ will be relevant relevant. This notion of a good pair is a purely algebraic-geometric notion, modulo the ground field $k$ supporting a possibly nontrivial derivation.

\begin{Definition} By a good pair $(V,W)$ over $k$ we mean a pair of $k$-irreducible affine varieties $V,W$ with $V\subseteq {\cal U}^{n}$ and $W\subset {\cal U}^{2n}$ for some $n$, such that
\newline
(i) $W\subseteq T_{\partial}(V)$,
\newline
(ii) $W$ projects generically on to $V$, and
\newline
(iii) For ${\bar a}$ generic  in $V$ over $k$, the fibre $W_{\bar a}$ is an {\em affine subspace} of $T_{\partial}(V)_{\bar a}$, namely is defined by a finite system of linear possibly inhomogeneous equations. 
\end{Definition}

\begin{Remark} Assume $(V,W)$ is a good pair over $k$. Let $({\bar a},{\bar b})$ be a generic point of $W$ over $k$.  Suppose $(b_{1},..,b_{m})$ is a transcendence basis for ${\bar b}$ over $k({\bar a})$. Then for each $i=1,..,n$, $b_{i}$ is of the form $s_{0}({\bar a}) + s_{1}({\bar a}) b_{1} + ... + s_{m}({\bar a})b_{m}$, where the $s_{j}({\bar a})$ are in $k({\bar a})$  (so as the notation suggests each $s_{j}$ is a $k$-rational function defined at ${\bar a}$). 
\end{Remark}

\begin{Lemma} Suppose that $(V,W)$ is a good pair over $k$, where $V\subseteq {\cal U}^{n}$.  Let ${\bar a}, {\bar b}$ and $m\leq n$ be as in Remark 2.2. Then there is a complete type $p({\bar x})$ over $k$ (in $DCF_{0}$)  which is axiomatized by:
\newline
(i)  $({\bar x}, {\bar x}')$ is a generic point over $k$ of $W$ (in the sense of $ACF$,
\newline
(ii) $(x_{1}',..,x_{m}',x_{1}'',..., x_{m}'',....,x_{1}^{(r)},...,x_{m}^{(r)},...)$ is algebraically independent over $k({\bar x})$. 
\newline
This complete type $p({\bar x})$ over $k$ (in $DCF_{0}$)  is determined uniquely by $(V,W)$ and we call it the ${\partial}$-generic type of $(V,W)$ over $k$. 
\end{Lemma}
\begin{proof} Suppose first that the set of formulas in (i), (ii) is consistent, so realized by ${\bar a}$ say. Write ${\bar b} = {\bar a}'$. So $({\bar a}, {\bar b})$ is as in Remark 2.2.
Then the quantifier-free $L^{-}$-type of $({\bar a}, {\bar b})$ over $k$ is determined. 
By assumption $(b_{1}',...,b_{m}')$ is algebraically independent over $k({\bar a}, {\bar b})$   so its quantifier-free $L^{-}$ type over $k({\bar a}, {\bar b})$ is determined. By Remark 2.2, for $  i = m+1,..,n$,  $b_{i}' = s_{0}({\bar a})' + s_{1}({\bar a})'b_{1} + .. s_{m}({\bar a})' b_{m} + s_{1}({\bar a})b_{1}'   + ... + s_{m}({\bar a})b_{m}'$.  
As the $s_{j}({\bar a})'$ are in $k({\bar a},{\bar b})$, it follows that the quantifier-free $L^{-}$-type of $(b_{1}',.....,b_{n}')$ over $k({\bar a},{\bar b})$ is determined.  This extends in the obvious way to the quantifier-free $L^{-}$-type of $(b_{1}'',..,b_{m}'')$ over $k({\bar a}, {\bar b}, {\bar b}')$ etc,  showig the uniqueness of the type axiomatized by (i) and (ii) (assuming its consistency). 

\vspace{2mm}
\noindent
Let us now show the consistency of the formulas in (i) and (ii).  First note that directly from Fact 1.1, (i) is consistent. Moreover if $({\bar a}, {\bar a}')$ is a generic point of $W$ over $k$, then from our assumptions, $a_{1}',..,a_{m}'$ are algebraically independent over $k({\bar a})$. 
\newline
We restrict ourselves to showing just that                                                                                                                                                                                                                                                       \newline
(*)   (i) together with ``$x_{1}',..,x_{m}', x_{1}'', ..., x_{m}''$ is  algebraically independent over $k({\bar x})$" is consistent. 

Generalizing to higher derivatives of the $x_{1},..,x_{m}$ follows in a similar fashion. 
Let us start by fixing $({\bar a}, {\bar b})$ generic over $k$ in $W$. Consider $T_{\partial}(W)$, a subvariety of ${\cal U}^{4n}$ in variables ${\bar x}, {\bar y}, {\bar u}$ and ${\bar v}$. Consider the fibre $T_{\partial}(W)_{({\bar a},{\bar b})}$ over $({\bar a}, {\bar b})$.  The (dominant) projection from $W$ to $V$ induces a map $\pi:T_{\partial}(W)_{({\bar a},{\bar b})} \to T_{\partial}(V)_{\bar a}$, taking $({\bar u}, {\bar v})$ to ${\bar u}$.  The map is a ``torsor" for the linear map $T(W)_{\bar a, \bar b} \to T(V)_{\bar a}$, so each fibre has the same dimension, which is clearly $m$.  By Lemma 1.6 of \cite{Pierce-Pillay}, $\pi$ is surjective, in particular ${\bar b}$ is in the image of $\pi$.  Choose ${\bar c}$ generic (over all the data) in the fibre above ${\bar b}$. By what we have just said $tr.deg({\bar c}/k({\bar a},{\bar b})) = m$, whereby clearly $c_{1},..,c_{m}$ are algebraically independent over 
$k({\bar a}, {\bar b})$.  Let $Z$ be the ($k$-irreducible) variety over $k$ of which $({\bar a}, {\bar b}, {\bar b}, {\bar c})$ is the $k$-generic point.  Hence $Z$ is a subvariety of $T_{\partial}(W)$ which projects generically onto $W$.  By Fact 1.1, there is a $k$-generic point of $Z$ of the form $({\bar d}, {\bar d}')$.  Now the $L^{-}$-type of $({\bar d}, {\bar d}')$ equals the $L^{-}$-type of $({\bar a}, {\bar b}, {\bar b}, {\bar c})$
So, without loss of generality ${\bar d} = ({\bar a}, {\bar b})$ and ${\bar d}' = ({\bar b}, {\bar b}')$.  So ${\bar b} = {\bar a}'$ and ${\bar d}' = {\bar a}''$.   Now $({\bar a}, {\bar a}')$ is a $k$-generic point of $W$, whereby $a_{1}',..,a_{m}'$ is algebraically independent over $k({\bar a})$. And we have just seen that $a_{1}'',..,a_{m}''$ is algebraically independent over $k({\bar a}, {\bar a}')$.  Hence we have consistency of the expression in (*), as required. 
\end{proof}

\begin{Lemma} For any complete type $p({\bar x}) = tp({\bar a}/k)$, there is $r\geq 0$ such that $p^{(r)} = tp(({\bar a},{\bar a}',...,{\bar a}')/k)$ is the $\partial$-generic type of some good pair $(V,W)$ over $k$. 
\end{Lemma}
\begin{proof} Let $r\geq 0$ be such that $tr.deg({\bar a}^{(r)}/k({\bar a},...,{\bar a}^{(r-1)}) = 
\newline
tr.deg({\bar a}^{(i)}/k({\bar a},...,{\bar a}^{(i-1)})$ for all $i\geq r$. $r$ exists because the integers $d_{n} = tr.deg({\bar a}^{(n)}/k({\bar a},....,{\bar a}^{(n-1)})$ are decreasing and bounded above by the length of the tuple ${\bar a}$. 
Equivalently choose a differential transcendence basis $a_{1},..,a_{d}$ over $k$ for the tuple ${\bar a}$. Then choose $r$ such that ${\bar a}^{(r)}$ is in the algebraic closure of $k({\bar a},...,{\bar a}^{(r-1)}, a_{1}^{(r)},..,a_{d}^{(r)})$.  Let us now rebaptize $({\bar a},...,{\bar a}^{(r)})$ as ${\bar b}$. Let $V$ be the $k$-irreducible variety over $k$, of which ${\bar b}$ is the $k$-generic point. And let $W$ be the same thing for  $({\bar b}, {\bar b}')$. Then $W$ projects dominantly to $V$, and $W$ is a subvariety of $T_{\partial}(V)$, as is easilly checked.  Note that $W_{\bar b}$ is a generic fibre of $W\to V$ and ${\bar b'}$ is a generic point of that fibre.
Now every coordinate of the tuple ${\bar b}'$ is either already in ${\bar b}$ or is among $a_{1}^{(r+1)},...,a_{n}^{(r+1)}$  (where ${\bar a} = (a_{1},..,a_{n})$).  Let $d$ be the integer mentioned  above.  Let ${\bar b}_{0}$ be  
$({\bar a}, {\bar a}',...,{\bar a}^{(r-1)})$. 

Fix $i> d$, with $i\leq n$.  So $a_{i}^{(r)}$ is in the algebraic closure of $k({\bar b}_{0}, a_{1}^{(r)},..,a_{d}^{(r)})$, and let this be witnessed by an irreducible polynomial $P(x_{1},..,x_{d},x)$ over $k({\bar b}_{0})$  (bearing in mind that $a_{1}^{(r)},..,a_{d}^{(r)}$ is algebraically independent over $k({\bar b}_{0})$).  Then applying $\partial$ to $P(a_{1}^{(r)},..,a_{d}^{(r)},a_{i}^{(r)}) = 0$,  we see that 
\newline
$\sum_{j=1,..,d}(\partial P/\partial x_{j})(a_{1}^{(r)},..,a_{d}^{(r)},a_{i}^{(r)})a_{j}^{(r+1)}  +  (\partial P/\partial x) (a_{1}^{(r)},..,a_{d}^{(r)},a_{i}^{(r)})a_{i}^{(r+1)} + 
\newline
P^{\partial}(a_{1}^{(r)},..,a_{d}^{(r)},a_{i}^{(r)}) = 0$. 
\newline
Hence $(a_{1}^{(r+1)}, ..,a_{d}^{(r+1)}, a_{i}^{(r+1)})$ satisfies a nontrivial linear equation over $k({\bar b})$. As $a_{1}^{(r+1)},..,a_{d}^{(r+1)}$ are algebraically independent over $k({\bar b})$, it follows that the fibre $W_{\bar b}$ is an affine subspace of $T_{\partial}(V)$, as required.
\end{proof}

Finally we discuss the effective content of the above.  In fact it is a little bit more: complete types are determined by certain distinguished formulas. Moreover there is an effective list of the distinguished formulas, and an effective means of passing from the formula to the associated complete type. 

Now assuming that $k$ is countable and that $DCF_{0}$ with constants for elements of $k$ is decidable, then it follows from Lemmas 2.3 and 2.4 that for each $n$ the  collection of complete types over $k$ in finitely many variables, is effectively listable, namely there is a list $p_{1}, p_{2},.....$ of complete $n$-types over $k$ such that $\{(i,j)$: the $jth$ formula is in $p_{i}\}$ is recursive.  In particular taking $k = \Q$ we have an effective listing of complete $n$-types of $DCF_{0}$.

More precisely, fixing $n$, we consider good pairs $(V,W)$ over $k$ such that for some $r$, $V$ is a subset of affine $nr$ space, and so $W$ is a subset of affine $2nr$-space with coordinates say $({\bar x}_{1},..., {\bar x}_{k}, {\bar u}_{1},..,{\bar u}_{k})$ (where the ${\bar x}_{i}$ and ${\bar u}_{i}$ are $n$-tuples), where we also impose (on $W$)  that ${\bar u}_{1} = {\bar x}_{2}$,...,
${\bar u}_{r-1} = {\bar x}_{r}$.  The collection of such good pairs over $k$ is recursive. By Lemma 2.3 the $\partial$-type of such a good pair $(V,W)$, now considered as a complete $n$-type (rather than $nr$-type) is recursive uniformly in the good pair. And by Lemma 2.4 every complete $n$-type over $k$ arises this way.

\vspace{5mm}
\noindent
One could also ask about an analogous effective listing of complete quantifier-free types in $ACFA$ (the model companion of the the theory difference fields), fixing the characteristic if one wants.  In fact, as pointed out to us by Michael Wibmer,  Cohn's theory of difference kernels \cite{Cohn} is the analogue of the theory of good pairs that we have discussed. There is again a notion of ``generic prolongation" ($\sigma$-generic quantifier-free type) attached to a difference kernel, but now there are maybe more than one, although only finitely many such $\sigma$-generic types.  This is enough to give an effective enumeration of the quantifier-free types in the style of the current paper.  On the other hand,  As Dave Marker pointed out to us, the finite injury argument for $DCF_{0}$ based on the ACC for radical differential ideals, adapts to the difference context, via the ACC on perfect difference ideals. 

Let us finally remark that in the differential case, the effective enumeration of complete types is the same thing as an effective enumeration of prime differential ideals, and there are various treatments of this in the literature, such as computing the prime decomposition of a radical differential ideal (see \cite{Golubitsky}).

\end{document}